\documentclass[12pt]{amsart}
\usepackage{amsmath,amsthm,amsfonts,amssymb,eucal}

\renewcommand {\a}{ \alpha }

\newcommand{\y}{\eta}
\newcommand{\e}{\epsilon}

\newcommand{\D}{\Delta}
\newcommand{\s}{\sigma}
\renewcommand{\l}{\lambda}

\newcommand{\z}{\zeta}

\newcommand{\p}{\partial}

\newcommand{\Om}{\Omega}

\newcommand{\R}{ \mathbb R}

\newcommand{\N}{ \mathbb N}

\newcommand{\CL}{\mathcal L}

\newcommand{\CC}{\mathcal C}

\newcommand{\CH}{\mathcal H}

\newcommand{\CR}{\mathcal R}

\newcommand {\GA}{\mathfrak A}
\newcommand {\GF}{\mathfrak F}

\newcommand {\GH}{\mathfrak H}

\newcommand {\GZ}{\mathfrak Z}

\newcommand {\ga}{\mathfrak a}
\newcommand {\gd}{\mathfrak d}

\newcommand {\ba}{\mathbf a}

\newcommand {\bd}{\mathbf d}

\newcommand {\bq}{\mathbf q}

\newcommand {\BH}{\mathbf H}

\newcommand {\BQ}{\mathbf Q}
\newcommand {\BP}{\mathbf P}

\newcommand {\BS}{\mathbf S}
\newcommand {\BT}{\mathbf T}

\newcommand {\BZ}{\mathbf Z}

\newcommand{\CA}{\mathcal A}

\newcommand{\wt}{\widetilde}
\newcommand{\wh}{\widehat}

\DeclareMathOperator{\re}{Re} \DeclareMathOperator{\dist}{dist}
\DeclareMathOperator{\res}{\restriction}

\newtheorem{thm}{Theorem}[section]

\newtheorem{prop}[thm]{Proposition}

\theoremstyle{definition}

\theoremstyle{remark}

\numberwithin{equation}{section}

\newcommand{\bsymb}{\boldsymbol}

\begin{document}
\title[Laplacian in a narrow strip, II]
{On the spectrum of the Dirichlet Laplacian in a narrow strip, II}
\author[L. Friedlander]{Leonid Friedlander}
\address{University of Arizona\\ Tucson, Arizona\\ USA}
\email{friedlan@math.arizona.edu}

\author[M. Solomyak]{Michael Solomyak}
\address{Department of Mathematics\\The Weizmann Institute of Science\\
Rehovot 76100\\Israel} \email{solom@wisdom.weizmann.ac.il}
\thanks{}
\date{}
\begin{abstract}
This is a continuation of the paper \cite{FS}.
We consider the Dirichlet Laplacian in a family of unbounded
domains $\{x\in\R,\ 0<y<\e h(x)\}$. The main assumption is that
$x=0$ is the only point of global maximum of the positive,
continuous function $h(x)$. We show that the number of eigenvalues lying
below the essential spectrum indefinitely grows as $\e\to 0$, and
find the two-term asymptotics in
$\e\to 0$ of each eigenvalue and the one-term asymptotics of the
corresponding eigenfunction. The asymptotic formulae obtained
involve the eigenvalues and eigenfunctions of an auxiliary ODE on
$\R$ that depends only on the behavior of $h(x)$ as $x\to 0$.

The proof is based on a detailed study of the resolvent of the operator
$\D_\e$.
\end{abstract}

\maketitle

\section{Introduction}\label{intro}
This paper is a continuation of the authors' work \cite{FS} where
we studied
the spectrum of the Dirichlet Laplacian $\D_{\e,D}$ in a narrow
strip
\begin{equation*}
\Om_\e=\{(x,y): x\in I,\ 0< y<\e h(x)\}.
\end{equation*}
The main objective in \cite{FS} was to understand the behavior of
eigenvalues as $\e\to 0$, and the main
assumption was that $I$ is a finite
segment and the continuous function $h(x)$ has on $I$ a single point of
global maximum.  We found the two-term asymptotics (in $\e\to 0$) of
each eigenvalue, and also the one-term asymptotics of each eigenfunction. Our
approach was based upon a careful study of the resolvent
\[\GA_{\e,D}^{-1}:=(\D_{\e,D}-K\e^{-2})^{-1}\]
with an appropriate choice
of the constant $K$; see eq. \eqref{0:3x} below.  Note that
we consider the Laplacian as a positive operator, so that
\[ \D\psi=-\psi''_{xx}-\psi''_{yy}.\]

Here we apply the same approach to two other problems of a similar nature.
One of them concerns the case $I=\R$, and the assumptions about $h(x)$ are
basically the same as in \cite{FS}, complemented by a mild additional
condition as $|x|\to\infty$. In another problem $I$ is a finite segment
(as it was in \cite{FS}), but the Dirichlet condition at the vertical
parts of $\p\Om_\e$ is replaced by the Neumann condition. This gives rise to
the operator $\D_{\e,DN}$ that was not discussed in \cite{FS}. In both cases,
and especially in the second one, the difference with the original problem
studied in \cite{FS} looks minor. However, some technical tools used there
no more apply in the new situation, and one has to look for appropriate
substitutes.

Note that the case of the Neumann boundary condition on the whole of
$\p\Om_\e$ is simpler than the case of the Dirichlet condition, see its
analysis in \cite {RS} and \cite{KZ}. The results in \cite {RS}, \cite{KZ}
concern a much wider class of domains than those in our paper \cite{FS}.

\vskip0.2cm

We believe that our results for unbounded $\Om_\e$, i.e. for $I=\R$, are of
some independent interest. They easily extend to the case when $I$ is a
half-line; we leave it to the reader. The results for the operator $\D_{\e,N}$
are of a more technical character. They are useful, since they allow one to
apply the Dirichlet -- Neumann bracketing for study of some other
problems. In particular,
in our problem for $I=\R$ this gives a simple way to obtain the asymptotics
of eigenvalues, avoiding an analysis of the resolvent; we show this is
section \ref{infeig}. One more problem concerns the case when $I=\R$ and the
function $h(x)$ is periodic (Laplacian in a thin
periodic waveguide). Here the spectrum of $\D_{\e,D}$ has the band -- gap
structure, and
we study the location of bands and give a bound for their widths;
it turns out that they decay exponentially as $\e\to 0$. We present our
corresponding results in a separate paper \cite{FS2}.
\vskip0.2cm

We complement these comments at the end of section \ref{set}, after
introducing necessary notations and formulating some of our main results.

\vskip0.2cm {\bf Acknowledgments}. This work was mostly motivated by the
questions
asked by participants of the workshop in Quantum Graphs, their Spectra and
Application (Cambridge, April 2007) after the talk given
by the second author.
We are grateful to all who asked questions. Especially, we would like to
thank Brian Davies and Leonid Parnovski for their suggestion to use
Dirichlet--Neumann bracketing in the case of the whole line.
The work was mostly done when both
authors visited the Isaac Newton Institute for Mathematical
Sciences in Cambridge, UK. We acknowledge the hospitality of the
Newton Institute. The first author was partially supported
by the NSF grant DMS 0648786.

\section{The case of finite segment: setting of the problem and formulation
of main
results}\label{set}
\subsection{Preliminaries.}\label{prel}
 Let $I=[-a,b]$ be a finite segment and $h(x)>0$
be a continuous function on $I$. We assume that

({\it i}) $x=0$ is the only point of global maximum of $h(x)$ on
$I$;

({\it ii}) The function $h(x)$ is $C^1$ on $I\setminus\{0\}$, and
in a neighborhood of $x=0$ it admits an expansion

\begin{equation}\label{1:1}
h(x)=\begin{cases} M-c_+x^m+O\bigl(x^{m+1}\bigr),\qquad & x>0,\\
M-c_-|x|^m+O\bigl(|x|^{m+1}\bigr),\qquad & x<0\end{cases}
\end{equation}
where $M,c_\pm>0$ and $m\ge 1$.

 \vskip0.2cm

We consider the Laplacian in $\Om_\e$, and we
always impose the boundary conditions
\begin{equation}\label{0:2}
\psi(x,0)=\psi(x,\e h(x))= 0.
\end{equation}
The conditions at $x=-a$ and $x=b$ can be either Dirichlet
 or Neumann, and we denote the corresponding operators $\D_{\e,D}$ and
$\D_{\e,DN}$, respectively; in \cite{FS} the operator $\D_{\e,D}$
was denoted as $\D_{\e}$. For the sake of brevity, sometimes we
speak about the `$D$-problem' and the `$DN$-problem'.

\vskip0.2cm Denote
\begin{equation*}
H^{1,d}(\Om_\e)=\{\psi\in H^1(\Om_\e):\ \psi(x,0)=\psi(x,\e h(x))=0\}.
\end{equation*}
The conditions \eqref{0:2} imply
\begin{equation}\label{0:2p}
 \int_0^{\e h(x)} \psi'_y(x,y)^2dy\ge \frac{\pi^2}{\e^2 h^2(x)}
\int_0^{\e h(x)} \psi^2(x,y)dy;
\end{equation}
 here $\psi$ is a smooth real-valued function. The nature of the problem
allows us to work with such functions only. By \eqref{0:2p},
\begin{gather}
\int_{\Om_\e}|\nabla\psi|^2dxdy\ge\frac{\pi^2}{M^2\e^2}\int_{\Om_\e}
\psi^2dxdy,\qquad \forall\psi\in H^{1,d}(\Om_\e).\label{0:2a}
\end{gather}
It is convenient for us to work with the quadratic
form
\begin{equation}\label{0:3}
\ga_{\e}[\psi]=\int_{\Om_\e}\left(|\nabla\psi|^2-
\frac{\pi^2}{M^2\e^2}\psi^2\right)dxdy,\qquad \psi\in
H^{1,d}(\Om_\e).
\end{equation}
We denote the corresponding operators $\GA_{\e,D}$ and
$\GA_{\e,DN}$, depending on the boundary condition at $x=-a$ and
$x=b$. We supress  subscripts $D$ and $DN$ when our argument applies
to both operators.
On their respective domains they act as
\begin{equation}\label{0:3x}
\GA_\e\psi=\D\psi-\frac{\pi^2}{M^2\e^2}\psi.
\end{equation}

\subsection{Limiting behavior of eigenvalues and
eigenfunctions}\label{limb}
It turns out that under the conditions ({\it i}), ({\it ii})
 this behavior
is determined by the operator on $L^2(\R)$ given by
\begin{equation}\label{1:s2}
    \BH=-\frac{d^2}{dx^2}+q(x) ,\qquad q(x)=
\begin{cases}
2\pi^2 M^{-3} c_+ x^m,\ x>0,\\
2\pi^2 M^{-3} c_- |x|^m,\ x<0.\end{cases}
    \end{equation}
The spectrum of $\BH$ is discrete and consists of simple
eigenvalues which we denote by $\mu_j$. The corresponding
eigenfunctions $X_j(x)$, normalized by the conditions
$\|X_j\|_{L^2(\R)}=1,\ X_j(x)>0$ for large $x>0$, decay as
$|x|\to\infty$ superexponentially fast.
If $m=2$ and $c_+=c_-=c$,
then $\BH$ turns into the harmonic oscillator.

 \vskip0.2cm

\begin{thm}\label{1:t0}
Let $I$ be a finite segment and $h(x)$ meet the conditions
({\it i}) and ({\it ii}). Then

1) the eigenvalues $\l_j(\e,D)$, $\l_j(\e,DN)$ of the operators
$\D_{\e,D}$, $\D_{\e,DN}$ have the same asymptotic behavior, namely
\begin{equation}\label{1:s1}
    \lim_{\e\to 0}\e^{2\a}\biggl(\l_j(\e)-\frac{\pi^2}
{M^2\e^2}\biggr)=\mu_j,
\end{equation}
where
\begin{equation}\label{1:s20}
\a=2(m+2)^{-1}.
\end{equation}

2) For the normalized eigenfunctions $\Psi_j(\e,D;x,y)$,
$\Psi_j(\e,DN;x,y)$ of the operators
$\D_{\e,D}$, $\D_{\e,DN}$ we have, with an appropriate choice of sign:
\begin{equation*}
    \lim_{\e\to 0}\int_{\Om_\e}\left(\Psi_j(\e;x,y)-
\frac{\sqrt2}{\sqrt{\e^{1+\a}h(x)}}X_j(x\e^{-\a})\right)^2dxdy= 0.
\end{equation*}
\end{thm}
For the $D$-case this is the result of theorems 1.1 and 1.4 in
\cite{FS}. For the $DN$-case the results are new.

\vskip0.2cm

 Our proof of theorem \ref{1:t0} is based upon the study of
the operator family $\GA_\e^{-1}$ as $\e\to 0$. It consists of two
steps. Firstly, we reduce the original problem to the one for an
auxiliary ordinary differential operator $\BQ_\e$ acting on
$L^2(I)$. Secondly, we show that the operators
$(\e^{2\a}\BQ_\e)^{-1}$ approach a family of operators on
$L^2(\R)$ that are unitarily equivalent to, and hence isospectral
with, the operator $\BH^{-1}$. In the next two subsections we
describe these steps and formulate the corresponding results.
Their proofs are given in section \ref{dn}; the
general scheme is explained in section \ref{scheme}. In section \ref{inf}
we show that under mild additional assumptions about the behavior
of $h(x)$ as $|x|\to\infty$ this scheme applies also to the case
$I=\R$. In the concluding section \ref{infeig} we prove that these
conditions can be simplified even further, if one is interested
only in the behavior of the eigenvalues.

\subsection{Reduction of dimension}\label{red}
 In $L^2(\Om_\e)$ we take the subspace $\CL_\e$ that consists of
functions
\begin{equation*}
\psi(x,y)=\psi_{\e,\chi}(x,y)=\chi(x)\sqrt{\frac2{\e h(x)}}\,
\sin\frac{\pi y}{\e h(x)}.
\end{equation*}
The mapping
\begin{equation}\label{1:2z}
    \bsymb\Pi_\e:\chi\mapsto\psi_{\e,\chi}
\end{equation}
     is an isometric isomorphism of $L^2(I)$ onto
$\CH_\epsilon$, and  we  identify any operator $\BT$ on
$\CH_\e$ with the operator $\bsymb\Pi_\e^{-1}\BT\bsymb\Pi_\e$
acting on $L^2(I)$. Let $\chi\in H^{1}(I)$.
 A direct computation shows that
\begin{equation}\label{1:4a}
    \ba_{\e}[\psi_{\e,\chi}]=\bq_\e[\chi]:=
\int_I\left(\chi'(x)^2+W_\e(x)\chi^2(x)\right)dx,
\end{equation}
where
\begin{equation}\label{1:4g}
W_\e(x)=\frac{\pi^2}{\e^2}\left(\frac1{h^2(x)}-\frac1{M^2}\right)
+\biggl(\frac{\pi^2}3+\frac14\biggr)\frac{h'(x)^2}{h^2(x)}.
\end{equation}
The quadratic form $\bq_{\e}[\chi]$, considered on $H^1(I)$, is
positive definite and closed in $L^2(I)$. The same is true for its
restriction to $H^{1,0}(I)$. The corresponding self-adjoint operator on
$L^2(I)$ acts as
\begin{equation}\label{1:4h}
\BQ_\e\chi=-\chi''+W_\e(x)\chi,
\end{equation}
with the Dirichlet or the Neumann condition at $\p I$. When it is
necessary to reflect it in the notations, we denote these
operators by
$\BQ_{\e,D}$ and $\BQ_{\e,N}$ respectively.

\vskip0.2cm

Below $\CL^\e$ stands for the orthogonal complement of $\CL_\e$ in
$L^2(\Om_\e)$. Given a Hilbert space $\GZ$,
we write $\bsymb0_\GZ$ for the zero operator on $\GZ$.
\begin{thm}\label{1:t2}
Under the assumptions of theorem \ref{1:t0} one has
\begin{equation}\label{1:10}
\left\|\GA_{\e}^{-1}-\BQ_{\e}^{-1}\oplus\bsymb{0_{\CL^\e}}\right\|
=O(\e^{3\a}),\qquad \e\to 0
\end{equation}
where $\GA_\e$ is either of the operators $\GA_{\e,D}$ and
$\GA_{\e,DN}$, and $\BQ_{\e}=\BQ_{\e,D}$ in the $D$-case and
$\BQ_{\e}=\BQ_{\e,N}$ in the $DN$-case.
\end{thm}

For the $D$-case this is theorem 1.2 in \cite{FS}.

\subsection{From $\BQ_\e$ to $\BH$.}\label{red2}
Along with the operator $\BH$ defined in \eqref{1:s2}, let us
consider on $L^2(\R)$ the operator family
\begin{equation*}
 \BH_\e=-\frac{d^2}{dx^2}+\e^{-2}q(x);\quad \e>0.
\end{equation*}
In particular $\BH_1=\BH$. The substitution $x=t\e^\a$ shows that
for any $\e>0$ the operator $\e^{2\a}\BH_\e$ is unitary equivalent
to $\BH$, and hence, $\e^{2\a}\BH_\e$ is an isospectral family of
operators.
\begin{thm}\label{1:t1}
Let $\BQ_\e$ be either of the operators $\BQ_{\e,D}$ and
$\BQ_{\e,N}$. One has
\begin{equation*}
 \lim_{\e\to 0}\|(\e^{2\a}\BQ_\e)^{-1}\oplus\bsymb0_{L^2(\R\setminus I)}
-(\e^{2\a}\BH_\e)^{-1}\|= 0.
\end{equation*}
\end{thm}
For the $D$-case this is a reformulation of theorem  1.3 in
\cite{FS}, see eq. (1.9) therein.
\vskip0.2cm

 Note that theorems \ref{1:t2} and \ref{1:t1} give stronger a result
than theorem \ref{1:t0}. For instance, they imply that
the convergence
\[ \e^{-2\a}\left(\l_j(\e)-\frac{\pi^2}{M^2\e^2}\right)^{-1}\to\mu_j^{-1},\]
cf. \eqref{1:s1}, is uniform in $j$.

Derivation of theorem \ref{1:t0} from theorems \ref{1:t2} and \ref{1:t1}
was explained in \cite{FS}, and we do not reproduce it here.

\section{Proof of theorem \ref{1:t2}: general scheme}\label{scheme}
For the $D$-case the proof is given in \cite{FS},
section 3. For the $DN$-case the scheme remains the same, and we see it
useful to present it in an abstract form.

\vskip0.2cm

 Let $\GH_\e,\ 0<\e\le \e_0$, be a family of separable
Hilbert spaces. Keeping in mind our original problem, it is sufficient
to consider them real.
Let $\CH_\e\subset\GH_\e$ be a family of their
(closed) subspaces. We denote by $\CH^\e$ the orthogonal
complement of $\CH_\e$ in $\GH_\e$, so that
\begin{equation*}
\GH_\e=\CH_\e\oplus\CH^\e.
\end{equation*}
Below $\BP_\e$ and $\BP^\e$ stand for the orthogonal projections
in $\GH_\e$ onto the subspaces $\CH_\e$ and $\CH^\e$ respectively,
and for an arbitrary element $\psi\in\GH_\e$ we standardly write
\[ \psi_\e=\BP_\e \psi,\qquad \psi^\e=\BP^\e \psi.\]

For each $\e$, let $\ga_\e[\psi_1,\psi_2]$ be a
symmetric bilinear form in $\GH_\e$,
defined for $\psi_1,\psi_2$ lying in a dense domain $\gd_\e$. We suppose
that the corresponding quadratic form $\ga[\psi]:=\ga_\e[\psi,\psi]$ is
non-negative and closed. We write
$\GA_\e$ for the self-adjoint operator on $\GH_\e$, generated by
$\ga_\e$.

Suppose that
\begin{equation*}
\psi\in\gd_\e\ \Longrightarrow\ \psi_\e \in\gd_\e.
\end{equation*}
Then also $\psi\in\gd_\e\ \Longrightarrow\ \psi^\e \in\gd_\e$ and
the sets
\[\bd_\e:=\{\psi_\e:\psi\in\gd_\e\}, \qquad \bd^\e:=\{\psi^\e:\psi\in\gd_\e\}\]
are dense in $\CH_\e$ and in $\CH^\e$, respectively. Indeed, if,
say, $\wt \psi\in\CH_\e$ and $(\psi_\e,\wt \psi)=0$ for all
$\psi_\e\in\bd_\e$, then also $(\psi,\wt \psi)=0$ for all
$\psi\in\gd_\e$, and hence $\wt \psi=0$.

Let us consider the families $\ba_\e=\ga_\e\res\bd_\e$ and
$\ba^\e=\ga_\e\res\bd^\e$ of quadratic forms in $\CH_\e$ and in
$\CH^\e$. Both $\ba_\e$ and $\ba^\e$ are closed. We denote by
$\CA_\e$ and $\CA^\e$ the corresponding self-adjoint operators on
$\CH_\e$ and on $\CH^\e$.

\vskip0.2cm

 The quadratic form $\ga_\e$ decomposes as
\begin{gather}
\ga_\e[\psi]=\ba_\e[\psi_\e]+\ba^\e[\psi^\e]+2\ga_\e[\psi_\e,\psi^\e],
\qquad \psi\in\gd_\e.
\label{10:4x}
\end{gather}

Now, we make the following assumptions about the behavior of each
term in \eqref{10:4x}.
\begin{gather}
\ba_\e[\psi_\e]\ge c(\e)\|\psi_\e\|^2,\qquad \forall \psi_\e\in\bd_\e,
\qquad c(\e)\ge c_0>0;\label{10:5}
\end{gather}
\begin{gather}
\ba^\e[\psi^\e]\ge p(\e)\|\psi^\e\|^2,\qquad \forall \psi^\e\in\bd^\e;
\label{10:6}\\
\qquad p(\e)\to\infty, \qquad c(\e)=O(p(\e));\label{10:6x}
\end{gather}
\begin{gather}
 |\ga_\e[\psi_\e,\psi^\e]|^2\le
q^2(\e)\ba_\e[\psi_{\e}]\ga_\e[\psi^\e],\qquad \forall
\psi\in\gd_\e,\qquad q(\e)\to 0.\label{10:7}
\end{gather}

\begin{prop}\label{10:lem}
Let the conditions \eqref{10:5} -- \eqref{10:7} be satisfied. Then
for $\e$ small enough the operator $\GA_\e$ is positive definite,
and
\begin{equation}\label{10:9}
\|\GA_\e^{-1}-\CA_\e^{-1}\oplus\bsymb0_{\CH^\e}\|\le
p(\e)^{-1}+Cq(\e)c(\e)^{-1}.
\end{equation}
\end{prop}
\begin{proof}
Along with $\ga_\e[\psi]$, let us consider its diagonal part, i.e.
the quadratic form
\[\wh\ga_\e[\psi]=\ba_\e[\psi_\e]+\ba^\e[\psi^\e],\qquad \psi\in\gd_\e.\]
By \eqref{10:5} and \eqref{10:6}, $\wh\ga_\e$ is positive
definite. It is also closed, since its both components are closed.
Let $\wh\GA_\e$ be the corresponding self-adjoint
operator on $\GH_\e$. By \eqref{10:6x}, we have for small $\e$:
\begin{equation*}
\|\wh\GA_\e^{-1}\|\le Cc(\e)^{-1},
\end{equation*}
with some constant $C>0$. Besides, \eqref{10:7} implies
\[ |\ga_\e[\psi_\e,\psi^\e]|\le q(\e)\wh\ga_\e[\psi]\]
and hence,
\[|\ga_\e[\psi]-\wh\ga_\e[\psi]|=
2|\ga_\e[\psi_\e,\psi^\e]|\le2q(\e)\wh\ga_\e[\psi].\]
For $\e$ small enough, so that $q(\e)\le 1/4$, this implies
\[1/2\,\wh\ga_\e[\psi]\le\ga_\e[\psi]\le3/2\,\wh\ga_\e[\psi],\qquad\forall
\psi\in\gd_\e.
\]
Hence, for such $\e$ the operator $\GA_\e$ is positive definite
(rather than only non-negative, as it was originally assumed), and
\begin{equation*}
\|\GA_\e^{-1}\|\le2C{c(\e)}^{-1}.
\end{equation*}
Note also that \eqref{10:6} is equivalent to
\begin{equation}\label{10:8x}
    \|(\CA^\e)^{-1}\|\le p(\e)^{-1}.
\end{equation}

 Further, for any $\psi_1,\psi_2\in \gd_\e$ and
$\e$ small enough, one has

\begin{gather*}
\left|(\GA_\e^{1/2}\psi_1,\GA_\e^{1/2}\psi_2)-
(\wh\GA_\e^{1/2}\psi_1,\wh\GA_\e^{1/2}\psi_2)
\right|=|\ga_\e[\psi_1,\psi_2]-\wh\ga_\e[\psi_1,\psi_2]|\\
=|\ga_\e[\psi_{1,\e},\psi_2^\e]+\ga_\e[\psi_1^\e,\psi_{2,\e}]|
\le q(\e)\left((\ba_\e[\psi_{1,\e}]\ga_\e[\psi_2^\e])^{1/2}+
(\ba_\e[\psi_{2,\e}]\ga_\e[\psi_1^\e])^{1/2}\right)\\
\le
 2q(\e)(\wh\ga_\e[\psi_{1}]\wh\ga_\e[\psi_2])^{1/2}\le
2\sqrt2\,q(\e)(\wh\ga_\e[\psi_{1}]\ga_\e[\psi_2])^{1/2}.
\end{gather*}
In the last formula we take $\psi_1=\wh\GA_\e^{-1}f,\ \psi_2=\GA_\e^{-1}g$;
here $f,g\in \GH_\e$ are arbitrary elements. Then
\begin{gather*}
|(\wh\GA_\e^{-1}f,g)-(\GA_\e^{-1}f,g)|\\
 \le 2\sqrt2\,q(\e)\bigl((\GA_\e^{-1}g,g)
(\wh\GA_\e^{-1}f,f)\bigr)^{1/2}\le 4Cq(\e)c(\e)^{-1}\|f\|\|g\|,
\end{gather*}
and therefore
\begin{equation}\label{10:z}
\|\GA_\e^{-1}-\wh\GA_\e^{-1}\|\le 4Cq(\e)c(\e)^{-1}.
\end{equation}
 Since
$\wh\GA_\e^{-1}=\CA_\e^{-1}\oplus(\CA^\e)^{-1}$, we conclude that
\[\|\wh\GA_\e^{-1}-\CA_\e^{-1}\oplus\bsymb0_{\CH^\e}\|=\|(\CA^\e)^{-1}\|\]
Together with \eqref{10:8x} and \eqref{10:z}, this leads to
\eqref{10:9}.
\end{proof}

\section{Proof of theorems \ref{1:t2} and \ref{1:t1} (the DN-case)}\label{dn}
\subsection{Preliminaries}\label{pr}
To prove theorem \ref{1:t2}, we use proposition \ref{10:lem}
with $\GH_\e=L^2(\Om_\e)$ and $\CH_\e=\CL_\e$. As in section \ref{red}, we
identify operators $\BT$ acting in $\CL_\e$ with their images
$\bsymb\Pi_\e^{-1}\BT\bsymb\Pi_\e$ acting in $L^2(I)$. Here
$\bsymb\Pi_\e$ is the isometry given by \eqref{1:2z}. The basic
bilinear form is
\begin{equation}\label{4:x}
\ga_\e[\psi_1,\psi_2]=\int_{\Om_\e}\left(\frac{\p \psi_1}{\p x}
\frac{\p \psi_2}{\p x}+ \frac{\p \psi_1}{\p y} \frac{\p \psi_2}{\p y}
-\frac{\pi^2}{M^2\e^2}\psi_1\psi_2\right)dxdy
\end{equation}
defined for $\psi_1,\psi_2\in \gd_\e:= H^{1,d}(\Om_\e)$. The corresponding
quadratic form is given by \eqref{0:3}.
For each
$\psi\in\bd_\e=H^{1,d}(\Om_\e)\cap\CL_\e$ there exists one and
only one $\chi\in H^1(I)$ such that $\psi_\e=\psi_{\e,\chi}$, and
then $\ba_\e[\psi_\e]=\bq_\e[\chi]$. It is immediate that the
operator $\CA_\e$ is nothing but $\BQ_{\e,DN}$, and hence, the
proof of theorem \ref{1:t2} for the DN-case reduces to the
proving the  inequalities \eqref{10:5} -- \eqref{10:7}
with appropriate constants $c(\e), p(\e)$ and $q(\e)$.

\subsection{Proof of \eqref{10:5}.}\label{42}
The inequality \eqref{10:5} is the only inequality, the proof
of which requires a new argument compared with \cite{FS}. There are
several ways to prove \eqref{10:5}. We choose a way that is based on a
remarkable result due to Birman, see
\cite{B1}, \cite{B2}. This result belongs to the general theory of
self-adjoint extensions of symmetric operators.

The inequality \eqref{10:5} can be rewritten as
\begin{equation*}
\bq_\e[\chi]\ge c\e^{-2\a}\|\chi\|^2,\qquad \chi\in H^1(I),
\end{equation*}
or equivalently,
\begin{equation}\label{4:5}
\|\BQ_{\e,N}^{-1}\|\le c^{-1}\e^{2\a}.
\end{equation}
The similar inequality for $\BQ_{\e,D}^{-1}$ was proved in
\cite{FS}, lemma 2.1. Therefore, it is sufficient to estimate the
norm of the operator $\BQ_{\e,N}^{-1}-\BQ_{\e,D}^{-1}$. For
technical reasons, it is more convenient to deal with the operator
\begin{equation*}
\BS_\e=(\BQ_{\e,N}+1)^{-1}-(\BQ_{\e,D}+1)^{-1}.
\end{equation*}
Its norm can be estimated by using lemma 3.1 in \cite{B2}. When applied
to the operators in question, it yields
\begin{equation}\label{1:m2}
\|\BS_\e\|=\max\left\{\CR(u):\ u\in H^1(I),\
-u''+(1+W_\e(x))u=0\right\},
\end{equation}
where
 \[ \CR(u)=\frac{\int_I|u|^2dx}
{\int_I\left(|u'|^2+(1+W_\e(x))|u|^2\right)dx}.\] The weak form of
the differential equation from \eqref{1:m2} is
\begin{equation}\label{1:c1b}
 \int_I\left(u'\phi'+(1+W_\e(x))u\phi\right)dx=0,\qquad
\forall\phi\in H^{1,0}(I).
\end{equation}
Take $\phi(x)=u(x)\z_{ab}(x)$; the function $\z_{ab}(x)$ is
described below. First, we fix a function
$\z\in C^\infty(0,2)$ such that
\[\z(t)=1,\ t\le 1;\ \z(t)=(2-t)^2,\ t>3/2;\ 0\le\z(t)\le 1\
{\text{everywhere}}.\]
Denote $ K=\max \z'(t)^2/\z(t)$, and set
\[\z_{ab}(x)=\begin{cases}b^2\z(2x/b),\ & x\in[0,b);\\
a^2\z(2|x|/a), & x\in(-a,0].\end{cases}\] \vskip0.1cm\noindent We
have $\z_{ab}\in C^\infty(I)$, $\z(-a)=\z(b)=0$,  and $\z_{ab}'(x)^2\le
4K\z_{ab}(x)$. The function $\phi=u\z_{ab}$ lies in $H^{1,0}(I)$,
and we conclude from \eqref{1:c1b} that
\begin{gather*}
2\int_I\left({u'}^2+(1+W_\e(x))u^2\right)\z_{ab} dx=-2\int_I u'u\z'_{ab}dx\\
\le \int_I {u'}^2\z_{ab} dx+ \int_I
u^2\frac{(\z'_{ab})^2}{\z_{ab}}dx \le \int_I {u'}^2\z_{ab}
dx+4K\int_{\wt I} u^2dx
\end{gather*}
where $\wt I=(-a,-a/2)\cup(b/2,b)$.
>From here we derive that
\[\int_I\left({u'}^2+2(1+W_\e(x))u^2\right)\z_{ab}dx\le
4K\int_{\wt I} u^2dx.\] Hence,
\[\int_{a/2}^{b/2}u^2 dx\le 2K\int_{\wt I} u^2dx.\]
The conditions on $h(x)$ imply the inequality
\[ W_\e(x)\ge \s\e^{-2}|x|^m,\qquad \forall\e>0,\ x\in I,\]
with some $\s>0$, see (2.2) in \cite{FS}. This yields
\[\CR(u)\le (1+2K)\frac{\int_{\wt I} u^2dx}
{\int\limits_{\wt I} W_\e(x)u^2dx}\le
\frac{1+2K}{\s(\min(a,b)/2)^m}\e^{2}.\] So,
$\|\BS_\e\|=O(\e^{2})$. From Hilbert resolvent formula we conclude
that also
\begin{equation}\label{10:y}
    \|\BQ_{\e,N}^{-1}-\BQ_{\e,D}^{-1}\|=O(\e^2).
\end{equation}
By \eqref{1:s20}, $\a\le2/3$ and hence,
$\|\BQ_{\e,D}^{-1}\|=O(\e^{2\a})$ yields the same estimate for
$\|\BQ_{\e,N}^{-1}\|$. This completes the proof of \eqref{4:5} and
hence, that of \eqref{10:5}, with $c(\e)=C\e^{-2\a}$.
\subsection{Proof of \eqref{10:6}}\label{34}
This inequality is the easiest to prove.
The inclusion $\psi\in\CH^\e$
means that
\begin{equation}\label{2:1}
\int_0^{\e h(x)}\psi(x,y)\sin\frac{\pi y}{\e h(x)}dy=0,\qquad
{\text{a.a.}}\ x\in I.
\end{equation}
In other words, the function $\psi(x,\cdot)$ is orthogonal to the
first eigenfunction of the operator $-u''_{yy}$ on the interval
$(0,\e h(x))$, with the Dirichlet boundary conditions at its ends.
Thus, every function $\psi\in\gd_\e\cap\CH^\e$ satisfies an estimate similar to
\eqref{0:2p}; its right hand side contains an additional factor $4$.
Therefore
\begin{equation*}
\|\psi\|^2\le \frac{M^2\e^2}{3\pi^2}\ga_\e[\psi],\qquad
\forall\psi\in \bd^\e
\end{equation*}
which means that \eqref{10:6} is satisfied with
\begin{equation}\label{2:1a}
p(\e)=3\pi^2M^{-2}\e^{-2}.
\end{equation}
In particular, \eqref{2:1a} implies \eqref{10:6x}.
If $m>1$ in \eqref{1:1} then the symbol $O$ in
\eqref{10:6x} can be replaced by $o$.

\subsection{End of the proof of theorem \ref{1:t2}.}\label{end2}
The proof of \eqref{10:7} repeats the proof
of the similar inequality in
\cite{FS}. Nevertheless, we outline the argument.

For $\psi\in\bd^\e$ we can integrate \eqref{2:1} by parts in $y$
and differentiate it in $x$. This results in the equalities that
hold for a.a. $x\in I$:
\begin{gather*}
\int\limits_0^{\e h(x)}\psi'_y(x,y)\cos\frac{\pi y}{\e h(x)}dy=0;\\
\int\limits_0^{\e h(x)}\psi'_x(x,y)\sin\frac{\pi y}{\e h(x)}dy=
\frac{\pi}{\e}\frac{h'(x)}{h^2(x)} \int\limits_0^{\e h(x)}
y\psi(x,y)\cos\frac{\pi y}{\e h(x)}dy.
\end{gather*}
Let $\psi\in\gd_\e$ and $\psi_\e=\psi_{\e,\chi}$ with some
$\chi\in H^1(I)$. The off-diagonal part of $\ga_\e[\psi]$ reduces
to the form
\begin{gather}
 \ga_\e[\psi_\e,\psi^\e]=\int_{\Om_\e}(\psi_{\e,\chi})'_x(\psi^{\e})'_xdxdy
\label{2:1z}\\
=\frac{\sqrt2\pi}{\e^{3/2}}\int_{\Om_\e} \wt h(x)\cos\frac{\pi
y}{\e h(x)}\left(\phi'\psi -\phi \psi'_x\right)ydxdy\notag
\end{gather}
where
\[ \wt h(x)=\frac{h'(x)}{h^2(x)},\qquad \phi(x)=\frac{\chi(x)}{h^{1/2}(x)}.\]
Now the estimate \eqref{10:7} with $q(\e)=C\e^\a$
follows by Cauchy-Schwartz inequality. \vskip0.2cm

Applying the estimate \eqref{10:9} to the operator $\GA_{\e,DN}$,
we conclude that
\[ \left\|\GA_{\e}^{-1}-\BQ_{\e}^{-1}\oplus\bsymb{0_{\CL^\e}}\right\|\le
C(\e^2+\e^{3\a});
\]
together with $\a\le2/3$, the last estimate implies \eqref{1:10}. This
completes the proof of theorem \ref{1:t2}.
\subsection{Proof of theorem \ref{1:t1}.}\label{end1}
For the $D$-case this is an equivalent reformulation of theorem
1.3 in \cite{FS}, see eq. (1.10) there. In the $DN$-case,
we apply
the inequality \eqref{10:y} which yields
\begin{gather*}
\|(\e^{2\a}\BQ_{\e,N})^{-1}\oplus\bsymb0_{L^2(\R\setminus I)}
-(\e^{2\a}\BH_\e)^{-1}\|\\
\le\|(\e^{2\a}\BQ_{\e,D})^{-1}\oplus\bsymb0_{L^2(\R\setminus I)}
-(\e^{2\a}\BH_\e)^{-1}\|+\e^{-2\a}\|\BQ_{\e,N}^{-1}-\BQ_{\e,D}^{-1}\|.
\end{gather*}
Both terms on the right tend to zero as $\e\to 0$: the first by
theorem 1.3 in \cite{FS} and the second by \eqref{10:y}, again
keeping in mind that $\a<2/3$. This completes the proof of theorem
\ref{1:t1}.

\section{The case $I=\R$: behavior of the resolvent }\label{inf}
In this and the next sections we study the case $I=\R$. The
spectrum of the Dirichlet Laplacian $\D_\e$ in $\Om_\e$ is now not
necessarily discrete, the structure of its essential component
$\s_{ess}(\D_\e)$ depends on the behavior of $h(x)$ at infinity.
In this section we study the case when $h(x)$
satisfies conditions ({\it i}) and ({\it ii}), and also the following
additional conditions:

({\it iii}) $\limsup\limits_{|x|\to\infty}h(x)< M$,
\vskip0.2cm
({\it iii$\,^\prime$}) $h'/h\in L^\infty(\R)$.

\noindent Our goal is to show that under these conditions analogues of theorems
\ref{1:t2} and \ref{1:t1} hold. A result similar to theorem \ref{1:t0} then
follows automatically,
though its formulation becomes a bit more
complicated because the operator $\D_\e$ can have non-empty essential spectrum.
\vskip0.2cm

The situation simplifies if one is interested only in the behavior of
eigenvalues.
In the next section \ref{infeig} we will give a short proof of an analogue
of statement 1) in
theorem \ref{1:t0}. For that purpose, the condition ({\it iii$\,^\prime$})
turns out to be  not necessary.

\subsection{Behavior of the resolvent: formulations and auxiliary results.}
\label{bres}
Let $I=\R$. The quadratic form \eqref{1:4a} is
positive definite and closed on the natural domain
\[\bd_\e=\{\chi\in H^1(\R): \bq_\e[\chi]<\infty\}.\]
As before, the corresponding self-adjoint operator $\BQ_\e$ is formally given
by \eqref{1:4h}. Note also that instead of two operators $\GA_{\e,D},\
\GA_{\e,DN}$ we have only one operator $\GA_\e$.

\begin{thm}\label{1:t2inf}
Let $I=\R$ and $h(x)$ satisfy the conditions
({\it i}) -- ({\it iii$\,^\prime$}).
Then the equality \eqref{1:10} holds.
\end{thm}

For the analogue of theorem \ref{1:t1} we do not need condition
({\it iii$\,^\prime$}).
\begin{thm}\label{1:t1inf}
Let $I=\R$ and $h(x)$ satisfies the conditions ({\it i}) -- ({\it iii}).
Then
\begin{equation}\label{1:12}
 \lim_{\e\to 0}\|(\e^{2\a}\BQ_\e)^{-1}
-(\e^{2\a}\BH_\e)^{-1}\|= 0.
\end{equation}
\end{thm}

As in section \ref{dn}, in the proofs of these theorems we rely
upon proposition \ref{10:lem}, taking $\GH_\e=L^2(\Om_\e)$ and
$\CH_\e=\CL_\e$. The basic bilinear form is again given by
\eqref{4:x}. We have $\CA_\e=\BQ_\e$. The latter operator acts on
$L^2(\R)$; it is generated by the quadratic form \eqref{1:4a}
defined on its natural domain. The most important thing is to prove
\eqref{1:12}.
Indeed, the conditions ({\it i}) --
({\it iii}) may lead to a function $W_\e(x)$ that is bounded; then theorem
2.16 from the book \cite{S}, which was the main ingredient of our
proof of theorem 1.3 in \cite{FS}, does not apply. We need an
appropriate substitute.  First of all, we prove the following
lemma which can be considered as a partial generalization of
theorem 2.16 in \cite{S}. By $\CC$ we denote the
ideal of all compact operators in the algebra of all bounded
operators.
\begin{prop}\label{3:thm}
Let $\BT\ge 0$ and $\BT_\e\ge 0,\ 0<\e\le\e_0$, be bounded self-adjoint
operators in a
separable Hilbert space $\GH$, and let $\BT_\e\to\BT$  strongly as $\e\to 0$.
Suppose also that there exists a bounded self-adjoint
operator $\BT_0$ such that $\BT_\e\le\BT_0$ for all $\e\le\e_0$, and
\[ \dist(\BT_0,\CC)=m,\qquad m\ge 0.\]
Then
\begin{equation*}
\limsup\limits_{\e\to 0}\|\BT_\e-\BT\|\le m.
\end{equation*}
\end{prop}
\begin{proof}
Fix $\y>0$ and find an operator $\BS=\BS^*\in\CC$ that satisfies
$\|\BT_0-\BS\|<m+\y$.
Let $\GF$ be a finite-dimensional subspace in $\GH$, such that
$\|\BS g\|\le\y\|g\|$ for any $g\perp\GF$. Then
\[\|\BT_0 g\|\le\|(\BT_0-\BS)g\|+\|\BS g\|\le(m+ 2\y)\|g\|,\qquad
\forall g\perp\GF.\]

Now, let $h\in\GH$ be an arbitrary element, and let $h=f+g$ where
$f\in\GF$ and $g\perp\GF$. Then
\[ \left((\BT_\e-\BT)h,h\right)=\left((\BT_\e-\BT)f,f\right)+
\left((\BT_\e-\BT)g,g\right)+2\re\left((\BT_\e-\BT)f,g\right).\]
Since the dimension of $\GF$ is finite,
strong convergence $\BT_\e\to\BT$ implies
\[\lim_{\e\to 0}\sup_{f\in\GF}\|(\BT_\e-\BT)f\|/\|f\|=0.\]
Therefore,
\[ |\left((\BT_\e-\BT)h,h\right)|\le|\left((\BT_\e-\BT)g,g\right)| +\y\|f\|^2+
2\y\|f\|\|g\|\]
if $\e$ is small enough. Further,
\[-(\BT_0 g,g)\le \left((\BT_\e-\BT)g,g\right)\le(\BT_0 g,g),\]
since both $\BT$ and $\BT_\e$ are non-negative operators. Therefore,
\[|\left((\BT_\e-\BT)g,g\right)|\le (\BT_0 g,g)\le (m+2\y)\|g\|^2.\]
Finally, this yields
\[|\left((\BT_\e-\BT)h,h\right)|\le (m+2\y)\|g\|^2+2\y\|f\|^2+\y\|g\|^2
\le (m+3\y)\|h\|^2.\]
Since $\y>0$ is arbitrary, the statement of the Proposition follows.
\end{proof}

\vskip0.2cm
Below, the symbol $\BZ_V$ denotes the Schr\"odinger operator on $L^2(\R)$
with the non-negative potential $V(x)$. We define the operator $\BZ_V$
via its quadratic form.
The next statement is a substitute for proposition 4.1 in \cite{FS}.

\begin{prop}\label{3:thm1}
Let $V(x)$ and $V_\e(x)$, $0<\e\le\e_0$, be non-negative measurable
functions
on $\R$, such that $V(x)\to\infty$ as $|x|\to\infty$ and
\begin{equation}\label{3:2}
V_\e(x)\to V(x)\ \text{as}\ \e\to 0,\qquad {\text{uniformly on compact sets.}}
\end{equation}
Suppose also that
\[ V_\e(x)\ge V^\circ_\e(x),\qquad \forall x\in\R,\ 0<\e\le\e_0,\]
where $V_\e^\circ(x)$ is another family of
measurable functions on $\R$, which is
monotone in $\e$:
\begin{equation}\label{3:3}
\e_1>\e_2\ \Longrightarrow\ V^\circ_{\e_1}(x)\le V^\circ_{\e_2}(x),
\qquad \forall x\in\R
\end{equation}
and
\begin{equation}\label{3:4}
c(\e):=\liminf\limits_{|x|\to\infty}V^\circ_\e(x)\to\infty
\qquad\text{as}\   \e\to 0.
\end{equation}
Then
\begin{equation}\label{1:9xx}
\|\BZ_{V_\e}^{-1}-\BZ_V^{-1}\|\to 0,\qquad  \e\to 0.
\end{equation}
\end{prop}
\begin{proof}
Denote
\[\BT=\BZ_V^{-1},\qquad \BT_\e=\BZ_{V_\e}^{-1}.\]
Let also $\BT_\e^\circ$ be the inverse to the
Schr\"odinger operator with the potential $V_\e^\circ(x)$.
The assumption \eqref{3:2} implies that $\BZ_{V_\e}u\to \BZ_{V}u$ for any
$u\in C_0^\infty(\R)$, and theorem 8.1.5 in \cite{K} guarantees
the strong convergence $\BT_\e\to\BT$.

Fix $\e^*\in(0,\e_0)$; then for $\e<\e^*$ the conditions of proposition
\ref{3:thm}
are satisfied with $\BT_0=\BT_{\e^*}^\circ$. From this proposition we
conclude that
\[\limsup\limits_{\e\to 0}\|\BT_\e-\BT\|\le 1/c(\e^*).\]
Taking $\e^*\to 0$, we arrive at \eqref{1:9xx}.
\end{proof}
\subsection{Proof of theorem \ref{1:t1inf}.}\label{112}
Introduce the potential
\begin{equation*}
V_\e(t)=\e^{2\a}W_\e(t\e^\a)
\end{equation*}
where $W_\e(x)$ is the function defined in \eqref{1:4g}. The assumption
\eqref{1:1} implies that
\[ V_\e(t)=q(t)+\pi^2\rho_1(t\e^\a)t^{m+1}\e^\a+\e^{2\a}v(t\e^\a)\]
where $\rho_1(x)$ is some function which is bounded on any finite interval
$(-a,a)$, cf. proof of theorem 1.3 in \cite{FS}. Hence, $V_\e(t)\to q(t)$
uniformly on compact subsets in $\R$.

It follows from the
assumptions ({\it i}) and ({\it iii}) that for $|x|\ge 1$ the function
 $\pi^2(h^{-2}(x)-M^{-2})$ is bounded below:
\begin{equation*}
\pi^2(h^{-2}(x)-M^{-2})\ge c_0,\qquad |x|\ge 1,
\end{equation*}
with some $c_0>0$. On $[-1,1]$ the inequality
$W_\e(x)\ge\s\e^{-2}|x|^m$ with some $\s>0$ is fulfilled,
which leads to the estimate
\begin{equation*}
W_\e(x)\ge \s_1\e^{-2}\min(|x|^m,1),\qquad \forall x\in\R,\ \e>0.
\end{equation*}
Hence,
\[  V_\e(t)\ge \s_1\min(\e^{-2+2\a+m\a}|t|^m,\e^{-2+2\a})=
\s_1\min(|t|^m,\e^{-2+2\a}).\]
The conditions \eqref{3:3} and \eqref{3:4} are satisfied if we take
\[V_\e^\circ(t)=\s_1\min(|t|^m,\e^{-2+2\a}).\]
So, theorem \ref{3:thm1} applies and yields
\[ \|\wh\BQ_\e^{-1}-\BH^{-1}\|\to 0,\]
where $\wh\BQ_\e$ is the Schr\"odinger operator with the potential $V_\e(t)$.
The substitution $t=x\e^{-\a}$ leads to \eqref{1:12}.

\subsection{Theorem \ref{1:t2inf}: outline of proof.}\label{anal}
The argument necessary for proving theorem \ref{1:t2inf} is quite similar
to the one in \cite{FS}, and also in section \ref{dn} of the present paper;
we only outline
it, concentrating on the few new moments.

\vskip0.2cm
First of all, note that
\[ \|(\e^{2\a}\BH_\e)^{-1}\|=\|\BH^{-1}\|=const,\]
and hence \eqref{1:12} yields
\[ \|\BQ_\e^{-1}\|\le C\e^{2\a}.\]
Equivalently, this means that \eqref{10:5} is satisfied with
$c(\e)=C^{-1}\e^{-2\a}$.
The condition \eqref{10:6} is satisfied with the same $p(\e)$ as in
\eqref{2:1a}; its proof does not change. Since $\a\le 2/3$, \eqref{10:6x}
also holds. It only remains to prove \eqref{10:7}. To this end, we repeat the
reasoning in section \ref{end2}. The only difference is that now for
estimating the integral in the right-hand side of \eqref{2:1z} we need
the condition ({\it iii$\,^\prime$}).

\section{The case $I=\R$: behavior of eigenvalues, a simple proof}\label
{infeig}
An analogue of theorem \ref{1:t0} follows from theorems \ref{1:t2inf}
and \ref{1:t1inf} in the same way as theorem \ref{1:t0} itself follows
from theorems \ref{1:t2} and \ref{1:t1}. However, there is a much simpler
and independent way to prove an analogue of statement 1) in
theorem \ref{1:t0}. It uses the Dirichlet -- Neumann bracketing, which is
possible, since theorem \ref{1:t0} is now in our disposal for both the $D$
and the $DN$ cases. For the proof, the condition
({\it iii$\,^\prime$}) is not needed.

\vskip0.2cm

In theorem \ref{1:t4} below $\nu(\e)$
stands for the bottom of $\s_{ess}(\D_\e)$, and we take
$\nu(\e)=\infty$ if $\D_\e$ has no essential spectrum. We denote
by $n(\e),\ n(\e)\le\infty$, the number of eigenvalues $\l_j(\e)<\nu(\e)$.
\begin{thm}\label{1:t4} Let $I=\R$.
If $h(x)$ satisfies the conditions ({\it i}), ({\it ii}) and ({\it iii}),
then $\nu(\e)\to\infty$ as $\e\to 0$. Moreover,
for small values of $\e$  the spectrum of $\D_\e$ below $\nu(\e)$ is non-empty,
$n(\e)\to\infty$ as $\e\to 0$, and for
each $j\in\N$ the equality \eqref{1:s1} holds;
$\mu_j$ are eigenvalues of the operator \eqref{1:s2}.
\end{thm}
\begin{proof}
Take any segment $\wh I=[-a,a]$. It follows from the conditions ({\it i})
and ({\it iii}) that
\[ M_a:=\sup\limits_{x\notin \wh I}h(x)<M.\]
Denote
\[ \Om_{a,\e}=\{(x,y)\in\Om_\e: x\in \wh I\},\qquad
\Om'_{a,\e}=\{(x,y)\in\Om_\e: x\notin \wh I\}.\]
Let $\D_{a,\e,D},\ \D'_{a,\e,D}$ be the Dirichlet Laplacians in
$\Om_{a,\e},\ \Om'_{a,\e}$ respectively; notations $\D_{a,\e,DN},\
\D'_{a,\e,DN}$ have the similar meaning. Then
\begin{equation*}
\D_{a,\e,DN}+\D'_{a,\e,DN} <\D'_\e<\D_{a,\e,D}+\D'_{a,\e,D}.
\end{equation*}

For $\Om'_{a,\e}$ an inequality similar to \eqref{0:2a}, with $M_a$ in place
of $M$, is satisfied; it implies that the spectra of both operators
$\D'_{a,\e,D}$ and $\D'_{a,\e,DN}$ lie above
the number $C(\e):=\frac{\pi^2}{M_a^2\e^2}$, which is bigger than
$\frac{\pi^2}{M^2\e^2}$.
It follows that $\nu(\e)\ge C(\e)$,
and the eigenvalues $\l_j(\e)=\l_j(\D_\e)$ lying below $C(\e)$
satisfy the two-sided inequality
\[\l_j(\D_{a,\e,DN})\le \l_j(\e)\le \l_j(\D_{a,\e,D}).\]
The asymptotics \eqref{1:s1} for the eigenvalues
$\l_j(\D_{a,\e,DN}), \ \l_j(\D_{a,\e,D})$ implies that the number
$\nu(\e)$ grows indefinitely  as $\e\to 0$, and that the eigenvalues
$\l_j(\e)$ have the same asymptotics given by \eqref{1:s1}.
 \end{proof}

Another problem
of the same nature was analyzed in section 6.1 of \cite{FS}, it concerns the
case when the segment $I$ is finite but the function $h(x)$ is allowed to
vanish at the ends of $I$.

\end{document}